\documentclass[12pt]{amsart}
\usepackage[utf8]{inputenc}
\usepackage[all]{xy}
\usepackage{amscd, amsmath, mathrsfs, amssymb, amsthm, amsxtra, bbding, epsfig, eucal, eufrak, graphicx, latexsym, url}
\usepackage[all]{xy}
\def \A {{\mathbb A}}

\def \C {{\mathbb C}}

\def \N {{\mathbb N}}

\def \Q {{\mathbb Q}}
\def \R {{\mathbb R}}

\def \Z {{\mathbb Z}}

\def \d {\,{\rm d}}
\def\re{{\Re e\,}}
\def\im{{\Im m\,}}

\def\le{\leqslant}
\def\leq{\leqslant}
\def\ge{\geqslant}

\theoremstyle{plain}
\newtheorem{theorem}{Theorem}

\newtheorem{lemma}{Lemma}[section]

\theoremstyle{remark}
\newtheorem{remark}{Remark}

\theoremstyle{definition}

\numberwithin{equation}{section}

\newcommand{\bpm}{\begin{pmatrix}}
\newcommand{\epm}{\end{pmatrix}}
\newcommand{\bsm}{\lt(\begin{smallmatrix}}
\newcommand{\esm}{\end{smallmatrix}\rt)}
\newcommand{\lt}{\left}
\newcommand{\rt}{\right}

\begin{document}

\vglue 0mm

\title[Local behavior of arithmetical functions]
{Local behavior of arithmetical functions with applications to automorphic $L$-functions}
\author{Yuk-Kam Lau, Jianya Liu \& Jie Wu}

\dedicatory{Dedicated to Kai-Man Tsang on the occasion of his 60th birthday}

\address{
Yuk-Kam Lau
\\
Department of Mathematics
\\
The University of Hong Kong
\\
Pokfulam Road
\\
Hong Kong}
\email{yklau@maths.hku.hk}
\address{
Jianya Liu
\\
School of Mathematics
\\
Shandong University
\\
Jinan, Shandong 250100
\\
China}
\email{jyliu@sdu.edu.cn}
\address{%
Jie Wu\\
CNRS\\
Institut \'Elie Cartan de Lorraine\\
UMR 7502\\
F-54506 Van\-d\oe uvre-l\`es-Nancy\\
France}
\curraddr{%
Université de Lorraine\\
Institut \'Elie Cartan de Lorraine\\
UMR 7502\\
F-54506 Van\-d\oe uvre-l\`es-Nancy\\
France
}
\email{jie.wu@univ-lorraine.fr}

\begin{abstract}
We derive a Voronoi-type series approximation for the local weighted mean of an arithmetical function that is associated 
to Dirichlet series satisfying a functional equation with gamma factors. 
The series is exploited to study the oscillation frequency with  a method of Heath-Brown and Tsang \cite{HBT94}. A by-product is another proof for the well-known result of no element in the Selberg class of degree $0<{\rm d}<1$. Our major applications include
the sign-change problem of the coefficients of automorphic $L$-functions for $GL_m$, which improves significantly some results of Liu and Wu \cite{LiuWu2015}. The cases of modular forms of half-integral weight and Siegel eigenforms are also considered. 
\end{abstract}

\subjclass[2000]{11M37, 11F66,11F30}
\keywords{Local weighted mean, Arithmetical functions,
Sign-changes}
\maketitle


\section{Introduction and main results}\label{sec_intro}

The functional equation satisfied by the Riemann zeta function is a prototype of the salient features of many interesting Dirichlet series  $\sum_n a_n \lambda_n^{-s}$, including the Selberg class of $L$-functions whose theory were greatly advanced recently. This general class were studied quite long time ago, for instance, Chandrasekharan and Narasimhan \cite{ChandrasekharanNarasimhan1962} obtained nice Voronoi-type series approximation for the Riesz means of the coefficients $a_n$. Often the Voronoi  series are effective for the study of many interesting properties such as the mean square formula, omega results and the occurrence of sign-changes in \cite{Lau1999} and \cite{LT2002} respectively.  The Riesz mean carries  a weight function whose smoothing effect leads to permissible convergence of approximations. However this weight (of the Riesz mean) is not suitable for localizing $a_n$ within a narrow range  and hence not for local means. In practice the local mean can be applied to depict the oscillations. 

The purpose of this paper is two-fold. Firstly we provide a Voronoi-type series to the local weighted mean of $a_n$, 
which is novel, and deduce the occurrence of oscillations over short intervals.  A cute consequence is an alternative argument for the emptiness of the Selberg class of degree $0<{\rm d}<1$. Secondly the oscillation result is applied to the sign-change problem of the Dirichlet series coefficients of automorphic $L$-functions. The sign problem drew good attentions (e.g. \cite{LauLiuWu2012, LiuQuWu2011, LiuWu2015, MatomakiRadziwill2015, MeherMurty2014}) and the widely used approach, especially for somewhat general situations, 
is based on the first and second moments. 
This approach, however, may not be efficient for high rank groups  due to the poor order estimate for the Rankin-Selberg $L$-function which is the key ingredient to evaluate the second moment. Here the oscillation is detected with the method in Heath-Brown 
and Tsang \cite{HBT94}. Our application in Section~\ref{sec_application} shows not only the utility of the Voronoi series 
for local weighted means but also the effectiveness of the method in \cite{HBT94}. 
Theorems~\ref{thm1}-\ref{thm2.5} of this section are for the general context, 
and the specific results for automorphic $L$-functions are in Section~\ref{sec_application}.

\subsection{Assumptions}\

We are concerned with the class of arithmetical functions considered as in 
\cite{ChandrasekharanNarasimhan1962, Berndt1971, Hafner1981} with a little variant. 
Let $\{a_n\}$ and $\{b_n\}$ be two sequences in $\C$,  where  $a_1$ and $b_1$ are nonzero, and $\{\lambda_n\}$ and $\{\mu_n\}$ be two strictly increasing positive number sequences tending to $\infty$. We assume the following conditions.
\begin{enumerate}
\item[(A1)] 
The two series 
$$
\phi(s) := \sum_{n\ge 1} a_n \lambda_n^{-s}
\qquad\text{and}\qquad
\psi(s) := \sum_{n\ge 1} b_n \mu_n^{-s}
$$ 
converges absolutely in some half-plane $\re s\ge \sigma^*$ for some  constant $\sigma^*>0$. 
\item[(A2)] 
Suppose $\alpha_\nu>0$ and $\beta_\nu, \widetilde{\beta}_\nu \in \C$ for $\nu=1, 2, \dots, d$, and let 
\begin{eqnarray*}
\Delta(s) := \prod_{1\le \nu\le d} \Gamma(\alpha_\nu s+\beta_\nu) 
\quad 
\mbox{and} 
\quad
\widetilde{\Delta}(s) := \prod_{1\le \nu\le d} \Gamma(\alpha_\nu s+\widetilde{\beta}_\nu).
\end{eqnarray*}
Then $\phi(s)$ and $\psi(s)$ satisfy the functional equation
\begin{eqnarray*}
\Delta(s) \phi(s) = \omega \widetilde{\Delta}(1-s) \psi(1-s)
\end{eqnarray*}
for some constant $\omega\in \C$ of modulus $1$.\addtocounter{footnote}{1}\footnote{Here we allow different sets of the archimedean parameters $\beta_\nu$ on the two sides in order to cover the Selberg class and  to specialize to the case $\delta=1$ by replacing $b_n$ with $b_n\mu_n^{1-\delta}$. So the generality is not lost.}
\item[(A3)] The function $\Delta(s) \phi(s)$ (and hence $\widetilde{\Delta}(s) \psi (s)$) extends meromorphically  to the whole $\C$. All singularities of $\Delta(s) \phi(s)$ fall inside the  disk  $\mathcal{D}$ enclosed by the anti-clockwise circle $\mathcal{C}$: $|z|=R$ for some $R>0$. As  $\Delta(s)^{-1}$ is entire, all poles of $\phi(s)$ lie in $\mathcal{D}$.  
Moreover, $\Delta(\sigma +\text{i}t)\phi(\sigma+\text{i}t)$ tends to $0$ uniformly in every vertical strip $b\le \sigma \le b'$ 
as $|t|\to \infty$. 
\end{enumerate}

\vskip 2mm

Define 
\begin{align*}
A 
& := \alpha_1 + \alpha_2 + \cdots + \alpha_d \,,
\\
B 
& := \beta_1 + \beta_2 + \cdots + \beta_d\,, 
\\
\widetilde{B}
& := \widetilde{\beta}_1 + \widetilde{\beta}_2 + \cdots + \widetilde{\beta}_d\,,
\end{align*}
and 
\begin{align*}
h 
& :=  \prod_{1\le \nu\le d} \bigg(\frac{2A}{\alpha_\nu}\bigg)^{2 \alpha_\nu},  
\\\noalign{\vskip -1mm}
a 
& := \frac1{4A} - \frac12 - \frac{B-\widetilde{B}}{2A} = -\vartheta + \text{i}\xi \, , 
\\ 
k 
& := \frac{d}2 -\frac14 - \frac{A+B+\widetilde{B}}2 =\kappa+\text{i}\eta\, .
\end{align*}
Thus $A>0$, $h>0$ while $B, \widetilde{B}, a, k\in \C$ (with $\vartheta,\xi,\kappa,\eta\in\R$). 

\begin{remark} 
In \cite{ChandrasekharanNarasimhan1962, Berndt1971, Hafner1981}, 
the gamma factors on both sides of the functional equation are equal, i.e. $\Delta(s)=\widetilde{\Delta}(s)$. 
However $\widetilde{\Delta}(s) := \overline{\Delta(\overline{s})}$ for the Selberg class. 
\end{remark}

\subsection{Set-up and results}\label{ssec_setting}\

We start with any fixed smooth function $\varphi_0$ supported on $[-1, 1]$ such that
\begin{enumerate}
\item[(i)] $0\le \varphi_0(u)\le 1$ on $[-1, 1]$, and 
\par
\vskip 1mm
\item[(ii)] $1\le \int_\R\varphi_0(u) \d u \le 2$. 
\end{enumerate}
Let $\delta>0$ be any quantity and $X$ a large number.   
Set  
$$
L := \delta^{-1} X^{1/(2A)}
\qquad
\text{and}
\qquad
\varphi(u) := \varphi_0((u-1)L).
$$ 
Then $\varphi$ is  smooth and compactly supported on $(0,\infty)$ with $0\le \varphi(u)\le 1$ on $[1-L^{-1}, 1+L^{-1}]$ 
and zero elsewhere. Its derivatives satisfy 
\begin{equation}\label{derivation_varphi}
\|\varphi^{(r)}\|_\infty \ll_{\varphi_0,r} L^r
\quad
\text{for $r\ge 0$}
\end{equation} 
and, moreover,
\begin{equation}\label{integral_varphi}
L^{-1} \le \int_0^\infty \varphi(u) \d u \le 2L^{-1}.
\end{equation} 
Denote by 
$$
\widehat{\varphi}(s) := \int_0^\infty \varphi(u) u^{s-1}\d u
$$ 
its Mellin transform. Then $\widehat{\varphi}$ is entire and 
\begin{equation}\label{UB_hatvarphi}
\widehat{\varphi}(s)\ll_r L^{r-1} (1+|s|)^{-r}
\quad\text{for all $r\ge 0$.}
\end{equation} 

Let $\sigma_1>\max(\sigma^*,R)$. Then the disk  $\mathcal{D}$ lies in the half-plane $\re s<\sigma_1$.   For $x\in [X, 4X]$, we have
\begin{eqnarray*}
\sum_{n\ge 1} a_n \varphi\bigg(\frac{\lambda_n}{x}\bigg) 
= \frac1{2\pi \text{i}}\int_{(\sigma_1)} \phi(s) \widehat{\varphi}(s)x^s \d s.
\end{eqnarray*}
Take a constant $\sigma_2$ such that 
$$
1-\sigma_2 >  \max\big\{\sigma^*, 1+R, -\alpha_1^{-1}\widetilde{\beta}_1, \dots, -\alpha_d^{-1}\widetilde{\beta}_d\big\}.
$$ 
Then $\mathcal{D}$ lies in the right side of $\re s=\sigma_2$ and all poles of $\widetilde{\Delta}(s)$ and  $\widetilde{\Delta}(s)\psi(s)$ lie in  $\re s<1-\sigma_2$. 
We shift the line of integral to $\re s=\sigma_2$ and apply the functional equation with a change of variable $s$ into $1-s$, and thus infer
\begin{equation}\label{Sphi}
\begin{aligned}
S_\varphi(x)
& :=  \sum_{n\ge 1} a_n \varphi\bigg(\frac{\lambda_n}{x}\bigg)
-M_\varphi (x)
\\\noalign{\vskip -1mm}
& \; = \frac{\omega}{2\pi \text{i}} \int_{(1-\sigma_2)} \frac{\widetilde{\Delta}(s)}{\Delta(1-s)}\psi(s) \widehat{\varphi}(1-s)x^{1-s} \d s
\\\noalign{\vskip 1mm}
& \;= \omega  \sum_{n\ge 1} \frac{b_n}{\mu_n} I(\mu_nx),
\end{aligned}
\end{equation}
where  
\begin{eqnarray}\label{M}
M_\varphi(x):= \frac1{2\pi \text{i}} \int_{\mathcal{C}} \phi(s) \widehat{\varphi}(s)x^s \d s 
\end{eqnarray}
and
\begin{eqnarray}\label{I}
I(y) := \frac1{2\pi \text{i}} \int_{(1-\sigma_2)} \frac{\widetilde{\Delta}(s)}{\Delta(1-s)} \widehat{\varphi}(1-s)y^{1-s} \d s.
\end{eqnarray}

Our first theorem is a Voronoi-type series approximation for the local weighted mean of $a_n$.

\begin{theorem}\label{thm1} 
Let $\delta>0$ be fixed and $X$ any sufficiently large number. Under the above assumptions {\rm (A1)}-{\rm (A3)} and notation,  
the function $S_\varphi(x)$ defined in \eqref{Sphi} satisfies  the asymptotic   expansion
\begin{equation}\label{main}
S_\varphi(x) 
= \frac{\omega e_0}{2Ah}  (hx)^{1-\vartheta+{\rm i}\xi} S_{\varphi,0}(x) 
+ O\big(L^{-1} x^{1-\vartheta-1/(2A)}\big)
\end{equation}
for any $x\in [X, 4X]$, where $e_0\in \C^\times$ is given by
\begin{equation}\label{e0}
e_0 
:= \sqrt{\frac2{\pi}} \prod_{\nu=1}^d\left(\frac{2A}{\alpha_\nu}\right)^{\alpha_\nu+\beta_\nu-\widetilde{\beta}_\nu}
\end{equation}
and 
$$
S_{\varphi,0}(x) 
:= \sum_{n\ge 1} \frac{b_n}{\mu_n^{\vartheta-{\rm i}\xi}}
\int_0^\infty \frac{\varphi(u)}{u^{\vartheta}} u^{{\rm i}\xi}\cos\big( (h\mu_n x u)^{1/(2A)} + k\pi\big) \d u.
$$
Besides we have 
$$
S_\varphi(x)\ll x^{1-\vartheta} L^{-1}
$$ 
for all $x\in [X, 4X]$. 
The implied constants in the $O$-term and $\ll$-symbol 
depend only on $\varphi_0$, $\delta$ and the parameters in {\rm (A1)} and {\rm (A2)}. 
\end{theorem}

With the help of Heath-Brown and Tsang \cite{HBT94}, 
we derive from Theorem \ref{thm1} the following result concerning sign changes.

\begin{theorem}\label{thm2}  
Under the assumptions {\rm (A1)}-{\rm (A3)}, there exist small $\delta>0$ and positive constants $c_0,c_+,c_-$ 
which depend at most on the parameters in {\rm (A1)} and {\rm (A2)} such that for all sufficiently large $x\ge X_0(\delta)$, the two inequalities
$$
\pm \re\bigg(\varsigma^{-1} \frac{S_\varphi(x_\pm )}{ ( \mu_1hx_\pm)^{{\rm i}\xi}}\bigg) >c_\pm x^{1-\vartheta} L^{-1} 
$$
hold for some $x_+, x_-\in [x-c_0 x^{1-1/(2A)}, x+c_0 x^{1-1/(2A)}]$, where $\varsigma := \omega e_0 b_1/|b_1|$ and 
the parameter $L$ in the set-up for $\varphi$ takes the value  $\delta^{-1} x^{1/(2A)}$.
\end{theorem}

\begin{remark} 
(i) If $\eta\neq 0$, we can get the same result for the imaginary part as well. 
\par
(ii) Obviously the two inequalities imply a sign-change of $\re(\varsigma^{-1} S_\varphi(t)/(\mu_1ht)^{{\rm i} \xi})$ in the short interval $[x-c_0x^{1-1/(2A)}, x+c_0x^{1-1/(2A)}]$. When $0<2A<1$, the interval  shrinks as $x\to \infty$, implying a highly oscillatory $x^{-{\rm i} \xi} S_\varphi(x)$. This is unlikely to happen in many circumstances, see Theorem~\ref{thm2.5}  below.
\par
(iii) 
In case $\phi(s)$ has no pole in $\mathcal{D}$ (see (A3)), $M_\varphi(x)\equiv 0$ and thus the positivity (resp. negativity) of $S_\varphi(t)$ implies $a_n>0$ (resp. $a_n<0$)  for some $|\lambda_n-x|\ll x^{1-1/(2A)}$. 
\end{remark}

The next is a consequence echoing the empty Selberg class for small degree $0<{\rm d} <1$ proved in \cite{CG}. Note ${\rm d}=2A$ here. 

\begin{theorem}\label{thm2.5}
Suppose $0<2A<1$ and  {\rm (A1)}-{\rm (A3)}. If either of the following conditions: 
\begin{enumerate}
\item[{\rm (a)}] $\phi(s)$ is entire,
\item[{\rm (b)}] $\xi=0$ and all poles of $\phi(s)$ are real,
\item[{\rm (c)}] all poles of $\phi(s)$ lie in the half-plane $\re s <\vartheta_0$ 
where $\vartheta_0:= 1-\vartheta+1/2A$,
\end{enumerate}
holds true, then for all large enough $n$, 
we have $\lambda_{n+1}- \lambda_n\ll \lambda_n^{1-1/(2A)}$,
where the implied $\ll$-constant is independent of $n$.
\end{theorem}
\begin{remark} Suppose $F(s)$ is belonged to the Selberg class $\mathcal{S}({\rm d})$ of degree $0<{\rm d}<1$. Then $\lambda_n=n$ for all $n\ge 1$, and the corresponding $\vartheta_0$ is $1/2+3/(2{\rm d})>1$. As  $F(s)$  has at most one pole at $s=1$, Condition (c) of Theorem~\ref{thm2.5} are  fulfilled, implying $1=\lambda_{n+1}-\lambda_n\ll  n^{-(1-{\rm d})/{\rm d}}$ for all large $n$, an absurdity. Hence $\mathcal{S}({\rm d})=\emptyset$ for  $0<{\rm d}<1$, cf. \cite{CG, Mol}.
\end{remark}

\vskip 8mm

\section{Applications}\label{sec_application}

We consider the respective $L$-functions arising from self-contragredient representations for $GL_m(\A_\Q)$, modular forms of half-integral weight and Siegel Hecke eigenforms. The latter two cases will be handled altogether.  

\subsection{Self-contragredient representations for $GL_m(\A_\Q)$}
Let $m\ge 2$ be an integer and 
let $\pi = \otimes \pi_p$ be an irreducible unitary cuspidal representation of $GL_m(\A_\Q)$.
We associate the local parameters $\{\alpha_\pi(p,j)\}_{j=1}^m \subset \C $ and $\{\mu_{\pi} (j)\}_{j=1}^m\subset \C$
respectively  to $\pi_p$ and $\pi_\infty$ by the Langlands correspondence. The automorphic $L$-function $L(s, \pi)$  attached to $\pi$ is entire and expands into 
\begin{equation*}
L(s,\pi)=\sum_{n\ge 1} \lambda_{\pi}(n) n^{-s} 
\end{equation*}
for $\re s>1$, where
$$
\lambda_{\pi}(n)
:= \prod_{p^\nu\| n} \sum_{\nu_1+\cdots+\nu_m = \nu} 
\prod_{1\le j\le m} \alpha_{\pi}(p, j)^{\nu_j}.
$$

Set $\phi(s)=\pi^{-ms/2} q_\pi^{s/2} L(s,\pi)$ and $\psi(s)=\pi^{-ms/2} q_{{\widetilde{\pi}}}^{s/2} L(s,\widetilde{\pi})$,
where $\tilde{\pi}$ is the contragredient of $\pi$ with $\{\overline{\mu_\pi (j)}: \,1\le j\le m\}$ 
as the set of parameters for $\widetilde{\pi}_\infty$ and $q_{\pi}$ is the arithmetic conductor of $\pi$. 
The functional equation in (A2) is satisfied with $\omega=1$, 
$$
\Delta(s) 
= \prod_{1\le j\le m} \Gamma \bigg(\frac{s+\mu_\pi(j)}2\bigg) 
\quad 
\mbox{and} 
\quad 
\widetilde{\Delta}(s)
= \prod_{1\le j\le m} \Gamma\bigg(\frac{s+\overline{\mu_\pi (j)}}2\bigg).
$$ 

It is known that
from  Kim \& Sarnak \cite{KimSarnak2003} ($2\le m\le 4$) 
and  Luo, Rudnick \& Sarnak \cite{LuoRudnickSarnak1999} ($m\ge 5$) that 
\begin{equation}\label{LRS}
|\alpha_\pi(p, j)|\leq p^{\theta_m} 
\qquad\text{and}\qquad
|\re \mu_\pi(j)| \leq \theta_m 
\end{equation}
for all primes $p$ and $1\le j\le m$, where 
\begin{equation*}
\theta_2 := \frac{7}{64},
\qquad
\theta_3 := \frac{5}{14},
\qquad
\theta_4 := \frac{9}{22},
\qquad
\theta_m := \frac{1}{2}-\frac{1}{m^2+1} 
\quad
(m\ge 5).
\end{equation*}
The Generalized Ramanujan Conjecture (GRC) asserts that
the inequalities in \eqref{LRS} hold for all primes $p$ and $1\le j\le m$ with 
\begin{equation*}
\theta_m=0.
\end{equation*}
Plainly \eqref{LRS} implies
\begin{equation*}
|\lambda_{\pi}(n)|\leq n^{\theta_m} d_m(n)
\end{equation*}
for all $n\ge 1$, where $d_m(n) := \sum_{n_1\cdots n_m=n} 1$.

When $\pi$ is self-contragredient, 
we have 
$$
\{\mu_\pi(j): \,1\le j\le m\}=\{\overline{\mu_\pi(j)}: \,1\le j\le m\}.
$$
Thus 
$\widetilde{\Delta}(s)=\Delta(s)$
and $\lambda_{\pi}(n)$ is real for all $n\ge 1$. Further,
$$
A = \tfrac{1}{2}m,
\qquad
B=\widetilde{B}\in \R,
\qquad
\xi=\eta=0,
\qquad
e_0= (2/\pi)^{1/2} (2m)^{m/2}.
$$ 
Let
$$
{\mathscr N}_{\pi}^{\pm}(x)
:= \sum_{\substack{n\le x\\ \lambda_{\pi}(n)\gtrless\,0}} 1.
$$
Liu \& Wu \cite{LiuWu2015} proved that 
\begin{equation}\label{LB_Npmx}
{\mathscr N}_{\pi}^{\pm}(x)
\gg_{\pi} x^{1-2\theta_m} (\log x)^{2/m-2}
\qquad
(x\ge x_0(\pi))
\end{equation}
holds unconditionally for $2\le m\le 4$ and 
holds under Hypothesis H of Rudnick-Sarnak \cite{RudnickSarnak1996} for $m\ge 5$. 
Moreover, this result was applied to evaluate the number of sign changes $\mathcal{N}_{\pi}^*(x)$ of the sequence $\{\lambda_{\pi}(n)\}_{n\ge 1}$  in the interval $[1, x]$. It is shown in \cite[Corollary 1]{LiuWu2015} that
\begin{equation*}
\mathcal{N}_{\pi}^*(x)\gg_{\pi} \log\log x
\qquad
(x\ge x_0(\pi))
\end{equation*}
unconditionally for $2\le m\le 4$ and under Hypothesis H for $m\ge 5$.

Now by Theorem~\ref{thm2} with $A=m/2$ (and Remark 2 (iii)), the sequence $\{\lambda_\pi(n)\}$ has a sign-change 
over the short interval $\mathcal{I}_x:=[x-cx^{1-1/m}, x+cx^{1-1/m}]$, i.e. $\lambda_\pi(n)\lambda_\pi(n')<0$ 
where $n,n'\in \mathcal{I}_x$, for all sufficiently large $x$ and for some absolute constant $c>0$. 
Thus we obtain readily the following.

\begin{theorem}\label{thm3}
Let $m\ge 2$ and let $\pi$ be a self-contragredient irreducible unitary cuspidal representation for $GL_m(\A_\Q)$. 
Then we have
\begin{equation*}
\mathcal{N}_{\pi}^*(x)\gg_{\pi} x^{1/m}
\end{equation*}
for $x\ge x_0(\pi)$.
In particular we have
\begin{equation}\label{LB_Npmx_2}
\mathscr{N}_{\pi}^{\pm}(x)\gg_{\pi} x^{1/m}
\end{equation}
for $x\ge x_0(\pi)$.
\end{theorem}
\begin{remark} 
(i)
Neither the Rankin-Selberg $L$-function nor the bound towards Ramanujan conjecture were used, 
showing the robustness of the method in \cite{HBT94}.
\par
(ii)
The inequality \eqref{LB_Npmx_2} improves Liu-Wu's \eqref{LB_Npmx} in two directions: 
the former is unconditional for all $m\ge 5$, and the exponent $1/m$ is better than $1-2\theta_m$ for $m\ge 3$.
\end{remark}

\subsection{Modular forms of half-integral weight and Siegel-Hecke eigenforms}
Let $N$ be a positive integer and $\chi$ a Dirichlet character mod $4N$. For any odd integer $k>1$, we define $\mathcal{S}_{k/2}^*(4N, \chi)$ to be the space of holomorphic cusp forms of half-integral weight $k/2$ and nebentypus $\chi$ for the congruence subgroup $\Gamma_0(4N)$ which are orthogonal to $\Theta$ where $\Theta$ is the span of unary theta series. Note that the space $\Theta$ is nonzero only  when $k=3$. 

Let $f\in \mathcal{S}^*_{k/2}(4N,\chi)$ and write $f(z)=\sum_{n\ge 1} \lambda(n) n^{(k-2)/4}e(nz)$. Associated to $f$ is an $L$-function defined as $L(s,f)=\sum_{n\ge 1} \lambda(n)n^{-s}$, which is entire and  satisfies the functional equation
\begin{eqnarray*}
\Phi(s) L(s,f) = \omega \Phi(1-s) L(1-s,g)
\end{eqnarray*}
where $\Phi(s) =\pi^{-s}N^{s/2} \Gamma(s+(k/2-1)/2)$ and  $g\in \mathcal{S}_{k/2}^*(4N, \overline{\chi}\left(\frac{4N}{\cdot}\right))$. The character $\left(\frac{4N}{\cdot}\right)$ is defined as in \cite[p.442]{Shi}. 
We are interested in, for nonzero $f\in \mathcal{S}^*_{k/2}(4N,\chi)$, the sequence $\{\lambda_r(n)\}_{n\ge 1}$ defined by 
\begin{eqnarray}\label{cr}
&& L(s,f)^r =\sum_{n\ge 1} \lambda_r(n)n^{-s}, \qquad \mbox{ where $r\in\N$ is arbitrary.}
\end{eqnarray}

Next we turn to Siegel eigenforms. Let $1\le m \le k$ be integers and $Sp(m,\Z)$ be the symplectic group. Define $\mathcal{M}_k(Sp(m,\Z))$ (resp. $\mathcal{S}_k(Sp(m,\Z))$) to be the space of holomorphic modular (resp. cusp) forms of weight $k$ for $Sp(m,\Z)$.  Hecke eigenforms mean the nonzero common eigenfunctions of the Hecke algebra.  Attached to each Hecke eigenform $F\in \mathcal{M}_k(Sp(m,\Z))$, one defines, cf. \cite{AS},  the {\it standard $L$-function} 
\begin{eqnarray*}
L_1(s,F) := \zeta(s)\prod_p \prod_{1\le j\le m} (1-\alpha_j(p)p^{-s})^{-1}(1-\alpha_j(p)^{-1}p^{-s})^{-1}
\end{eqnarray*}
where $\{\alpha_j(p)\}_{1\le j\le m}$ are the Satake parameters of $F$, and  
the {\it spinor $L$-function}
\begin{eqnarray*}
L_2(s,F) := \prod_p \prod_{0\le k\le m} \prod_{1\le j_1<\cdots <j_k\le m} (1-\alpha_0(p) \alpha_{j_1}(p)\cdots \alpha_{j_k}(p)p^{-s})^{-1}
\end{eqnarray*}
where $\alpha_0(p)^2 \alpha_1(p)\cdots \alpha_m(p)=1$. Here $\re s $ is taken to be sufficiently large.

We know from  \cite{Miz} the following. When $m\equiv 0$ mod $4$, the space $\mathcal{M}_k(Sp(m,\Z))$ may contain theta functions that are associated to symmetric positive definite even integral unimodular $m\times m$ matrices and polynomials satisfying some conditions. These theta functions generate a subspace $\mathcal{B}_k(Sp(m,\Z))$ in $\mathcal{M}_k(Sp(m,\Z))$ which is invariant under the action of Hecke algebra.  Set $\mathcal{B}_k(Sp(m,\Z))=\{0\}$ if no such theta function exists. Define $\mathcal{H}_k^*(Sp(m,\Z))$ to be the set of Hecke eigenforms $F$ in $\mathcal{S}_k(Sp(m,\Z))\setminus \mathcal{B}_k(Sp(m,\Z))$. The Euler product of $L_1(s,F)$ converges absolutely and uniformly for $\re s>m+1$.  Besides, the complete $L$-function 
\begin{eqnarray*}
\Lambda_1(s,F) :=\Gamma_\R(s+\delta_{2\nmid m}) \prod_{1\le j\le m} \Gamma_\C(s+k-j) L_1(s,F)
\end{eqnarray*}
is holomorphic and satisfies the functional equation 
$$
\Lambda_1(s,F)= \Lambda_1(1-s,F).
$$
Here $\delta_*=1$ if the condition $*$ holds, and $0$  otherwise, $\Gamma_\R(s)= \pi^{-s/2} \Gamma(s/2)$ and $\Gamma_\C(s)=2(2\pi)^{-s} \Gamma(s)$. 

For $F\in \mathcal{H}_k^*(Sp(m,\Z))$, we define  the sequence $\{\lambda_{1,F}(n)\}_{n\ge 1}$ by 
\begin{eqnarray}\label{dr}
 && L_1(s,F)= \sum_{n\ge 1} \lambda_{1,F}(n)n^{-s}  \quad \mbox{  and \quad set $r:=m+\frac12$.}
\end{eqnarray}

The case of spinor $L$-functions is less understood, but the recipe in automorphic representation theory would suggest  hypothetically a working ground, cf \cite{Sch}. Let
$$
\Lambda_2(s,F)
= \Gamma_\C(s)^{N/2} \prod_{\substack{
\boldsymbol{\varepsilon}\in \{\pm 1\}^m\\ \alpha(\boldsymbol{\varepsilon}) >0}} 
\Gamma_\C\big(s+\tfrac{1}{2}|k\alpha(\boldsymbol{\varepsilon}) - \beta(\boldsymbol{\varepsilon})|\big) \prod_{\substack{\boldsymbol{\varepsilon} \in \{\pm 1\}^m\\ \alpha(\boldsymbol{\varepsilon})=0,\, \beta(\boldsymbol{\varepsilon})>0}} \Gamma_\C\big(s+\tfrac{1}{2}\beta(\boldsymbol{\varepsilon})\big)
$$
where 
$
N := \#\{\boldsymbol{\varepsilon} = (\varepsilon_1, \dots, \varepsilon_m)\in \{\pm 1\}^m : 
\alpha(\boldsymbol{\varepsilon}) = \beta(\boldsymbol{\varepsilon})=0\}
$ 
with 
$$
\alpha(\boldsymbol{\varepsilon}) 
:= \varepsilon_1 + \varepsilon_2 + \cdots +\varepsilon_{m}
\qquad\text{and}\qquad 
\beta(\boldsymbol{\varepsilon})
:= \varepsilon_1+ 2\varepsilon_2 + \cdots + m\varepsilon_{m}.
$$

We {\it assume} the eigenform $F\in \mathcal{H}_k(Sp(n,\Z))$ satisfies the following hypothesis.
\begin{enumerate}
\item[{\bf (EF)}]  $\Lambda_2(s,F)$ extends to an entire function and satisfies the functional equation 
$$
\Lambda_2(s,F)= \omega \Lambda_2(1-s,F)\qquad (|\omega|=1).
$$
\end{enumerate}

Now let us  introduce  the sequence $\{\lambda_{2,F}(n)\}_{n\ge 1}$ for which
\begin{eqnarray}\label{er}
&& L_2(s,F)= \sum_{n\ge 1} \lambda_{2,F}(n)n^{-s}  \quad \mbox{ and \quad set $r:=2^{m-1}$.}
\end{eqnarray}

Similarly to Theorem~\ref{thm3}, we have the following result.
\begin{theorem}\label{thm4} Let $r$  and $\{a_r(n)\}_{n\ge 1}$ be the value and the sequence defined as in  (\ref{cr}), (\ref{dr}) and (\ref{er}), where the case (\ref{er}) is particularly conditional on Hypothesis {\bf (EF)}.   Suppose $a_r(n)\in \R$ for all $n$, and  $x$ is any sufficiently large number. Then $\{a_r(n)\}$ has a sign-change  as $n$ runs over the short interval $[x-cx^{1-1/(2r)}, x+cx^{1-1/(2r)}]$.  Moreover, the number of sign-changes in $\{a_r(n)\}_{1\le n\le x}$ is $\gg x^{1/(2r)}$, so is the number of the terms in $\{a_r(n)\}_{1\le n\le x}$ of the same sign.
\end{theorem}
\begin{remark} This extends the study in \cite{RSW} on the spinor $L$-functions of Siegel eigenforms of genus 2, i.e. $m=2$. 
\end{remark}
\vskip 8mm

\section{Preparation}\label{sec_preparation}

In this section, we establish three preliminary lemmas for Theorems \ref{thm1} and \ref{thm2}.

\begin{lemma}\label{lem1} 
Let $J\ge 1$ be an integer. We have
\begin{equation}\label{D}
\frac{\widetilde{\Delta}(s)}{\Delta(1-s)} 
= \sum_{0\le j\le J-1} e_j F_j(s) + F_0(s) \cdot O_J\big(|s|^{-J}\big)
\end{equation}
as $|s|\to \infty$, where $e_j\in\C$ are constants with $e_0$ given by \eqref{e0} and
$$
F_j (s) 
:= h^{-s} \Gamma\big(2A(s+a)-j\big) \cos\big(\pi A(s+a)+k\pi\big). 
$$
The implied $O$-constant depends only on the parameters in {\rm (A2)} and $J$. 
 \end{lemma}
\begin{proof} This follows from the Stirling formula for $\log \Gamma(s)$: for any constant $c\in \C$ and any $J\in \N$, 
\begin{eqnarray*}
\log \Gamma(s+c)= (s+c-\tfrac12)\log s - s +\tfrac12 \log 2\pi +\sum_{1\le j\le J-1} c_j s^{-j} +O\big(|s|^{-J}\big)
\end{eqnarray*}
as $|s|\to \infty$, uniformly for where $|{\rm arg}\, s|\le \pi -\varepsilon< \pi$, where the constants $c_j$ depend on $c$. 
Here the empty sum means 0 and the empty product means 1.

As in the proof of \cite[Lemma 1]{CN1963}, we obtain, for $\alpha>0$, $\beta$ and $\widetilde{\beta}\in \C$, 
\begin{align*}
\log \frac{\Gamma(\alpha s+\widetilde{\beta})}{\Gamma(\alpha(1-s)+\beta)}
& = \big(\alpha s+\widetilde{\beta}-\tfrac12\big)\log s - \big(\alpha (1-s)+\beta -\tfrac12\big)\log (-s) 
\\
& \hskip -8mm
+ 2(\alpha\log\alpha-\alpha)s  +(\widetilde{\beta}-\beta-\alpha)\log \alpha+ \sum_{1\le j\le J-1} c_j' s^{-j} +O\big(|s|^{-J}\big)
\end{align*}
for some constants $c_j'$ depending on $\alpha$, $\beta$ and $\widetilde{\beta}$. 

Consequently, 
\begin{equation}\label{logD}
\begin{aligned}
\log\frac{\widetilde{\Delta}(s)}{\Delta(1-s)}
& = \big(A s+\widetilde{B}-\tfrac{1}{2}d\big)\log s - \big(A (1-s)+B -\tfrac{1}{2}d\big)\log (-s) 
\\
& \quad
+ 2\Big(\sum_{1\le \nu\le d} \alpha_\nu\log \alpha_\nu-A\Big)s + f + \sum_{1\le j\le J-1} {d_j^{(1)}} s^{-j} + O\big(|s|^{-J}\big) 
\end{aligned}
\end{equation}
where and throughout this proof, ${d_j^{(1)}}$, ${d_j^{(2)}}$, \dots \, denote some constants, and 
$$
f := \sum_{1\le \nu\le d}(\widetilde{\beta}_\nu-\beta_\nu-\alpha_\nu)\log \alpha_\nu. 
$$

On the other hand, we have
\begin{eqnarray*}
\Gamma(2A(s+a))\cos (\pi A(s+a)+k\pi) = \frac{\pi \Gamma(2A(s+a))} {\Gamma(\frac12- A(s+a)-k)\Gamma(\frac12+A(s+a)+k)}
\end{eqnarray*}
and thus
\begin{equation}\label{logG}
\begin{aligned}
& \hskip -3mm
\log\big(\Gamma(2A(s+a))\cos (\pi A(s+a)+k\pi)\big) 
\\
& = \big(A (s+a)-k-\tfrac12\big)\log s +\big(A (s+a)+k\big)\log (-s) 
\\
& \quad
+ 2\big(A\log (2A)-A\big)s  +f'+ \sum_{1\le j\le J-1} {d_j^{(2)}} s^{-j} +O\big(|s|^{-J}\big), 
\end{aligned}
\end{equation}
where 
$$
f' := \big(2Aa-\tfrac12\big)\log(2A) +\tfrac12\log \tfrac{\pi}2= (\widetilde{B}-B-A)\log (2A) +\tfrac12\log\tfrac{\pi}2.
$$
In view of the values of $a$, $h$ and $k$, the difference  between (\ref{logD}) and (\ref{logG}) equals
$$
f-f' -s\log h + \sum_{1\le j\le J-1} {d_j^{(3)}} s^{-j} +O\big(|s|^{-J}\big).
$$
Clearly for some constants $d_j^{(4)}$, we can write
\begin{align*}
& \exp\Big(f-f' -s\log h + \sum_{1\le j\le J-1} d_j^{(3)} s^{-j} +O\big(|s|^{-J}\big)\Big)
\\
& \hskip 3mm
= e_0 h^{-s} \Big(1 + \sum_{1\le j\le J-1} d_j^{(4)} s^{-j} +O\big(|s|^{-J}\big)\Big)
\end{align*}
with $e_0=\text{e}^{f-f'}$ (as given by \eqref{e0}) and the empty sum means 0.

Besides, for each $1\le j\le J-1$, we have
\begin{eqnarray*}
\frac{1}{s^j} -\frac{(2A)^j}{\prod_{\ell=1}^j (2A(s+a)-\ell)} 
= \sum_{j+1\le j'\le J-1} d_{j'}^{(5)} s^{-j} +O\big(|s|^{-J}\big).
\end{eqnarray*}
Consequently, a successive application of this formula gives
\begin{align*}
\sum_{1\le j\le J-1} d_j^{(4)} s^{-j}
& = \frac{d_1^{(5)}}{(2A(s+a)-1)} + \sum_{2\le j\le J-1} d_j^{(5)} s^{-j} + O\big(|s|^{-J}\big)
\\
& = \sum_{1\le j\le 2} \frac{d_j^{(6)}}{\prod_{\ell=1}^j (2A(s+a)-\ell)} + \sum_{3\le j\le J-1} d_j^{(6)} s^{-j} + O\big(|s|^{-J}\big)
\\\noalign{\vskip 1mm}
& \, \cdots
\\
& = \sum_{1\le j\le J-1} \frac{d_j^{(J+3)}}{\prod_{\ell=1}^j (2A(s+a)-\ell)} + O\big(|s|^{-J}\big).
\end{align*}
For the difference between (\ref{logD}) and (\ref{logG}), the left-hand side is
$$
\log  \frac{\widetilde{\Delta}(s)}{\Delta( 1-s)} -\log \big(\Gamma(2A(s+a))\cos (\pi A(s+a)+k\pi)\big),
$$
and the right-hand side is
$$
\log \Big(h^{-s} \sum_{0\le j\le J-1} \frac{e_j}{\prod_{\ell=1}^{j}(2A(s+a)-\ell)} +h^{-s}O\big(|s|^{-J}\big)\Big).
$$
Thus,
$$
\frac{\widetilde{\Delta}(s)}{\Delta( 1-s)}
= \frac{\cos(\pi A(s+a)+k\pi)}{h^s} \sum_{0\le j\le J-1} \frac{e_j \Gamma(2A(s+a))}{\prod_{\ell=1}^{j}(2A(s+a)-\ell)} 
+ F_0(s)O\big(|s|^{-J}\big).
$$
Using the formula $\Gamma(s+1)=s\Gamma(s)$, it is easy to see that
\begin{align*}
\frac{\Gamma(2A(s+a))}{\prod_{\ell=1}^{j}(2A(s+a)-\ell)} 
& = \frac{\Gamma(2A(s+a)-1)}{\prod_{\ell=2}^{j}(2A(s+a)-\ell)} 
= \cdots 
= \Gamma(2A(s+a)-j).
\end{align*}
Inserting  into the preceding formula, we obtain the required \eqref{D}.
\end{proof}

\begin{lemma}\label{lem2} 
Let $\varepsilon>0$ be arbitrarily small. There exists $J'\in\N$ such that for all $J\ge J'$, 
\begin{align*}
I(y) 
& = \frac{y}{2A}\sum_{0\le j\le J-1} e_j  
\int_0^\infty (hyu)^{a-j/(2A)} \varphi(u)\cos\big((hyu)^{1/(2A)} + (k+\tfrac{1}{2}j)\pi\big)\d u 
\\
& \quad
+ O_{\varepsilon, J}\big(L^{-1} y^{1- \vartheta -(J-\frac12)/(2A)+\varepsilon}\big),
\end{align*}
where the implied $O$-constant depends only on $\varepsilon$, $J$, $\varphi_0$ and  the parameters in {\rm (A1)} and {\rm (A2)}. 
\end{lemma}

\begin{proof} 
Let $0<\varepsilon<1/(4A)$ and write $\theta_J(\varepsilon) := \vartheta- \varepsilon+(J-\frac12)/(2A)$. Choose $J'\in \N$ such that $\theta_{J'}(\varepsilon)>1-\sigma_2$ and consider $J\ge J'$. 
As $\widehat{\varphi}(s)$ decays rapidly (thanks to \eqref{UB_hatvarphi}), we can move the line of integration of $I(y)$ 
in \eqref{I} from $\re s=1-\sigma_2$ to the right line $ \re s= \theta_J(\varepsilon)$. 
Then we insert (\ref{D}) into (\ref{I}) and integrate term by term.  

Let $s=\sigma +\text{i}t$. 
We  have $F_0(s) \ll |t|^{2A(\sigma - \vartheta) -1/2}$, as $|t|\to \infty$,  in any vertical strip  $b<\sigma< b'$. 
By \eqref{UB_hatvarphi} with $r=1$, the rightmost term $F_0(s) O_J(|s|^{-J})$ of \eqref{D} contributes at most $O_{\varepsilon}(y^{1- \theta_J(\varepsilon)})$ to the integral $I(y)$.  

Next we evaluate the integral of every summand in the sum $\sum_{j=0}^{J-1}$,  
which can be expressed as $e_j y I_j$ where
$$
I_j := \frac1{2\pi \text{i}} \int_{(\theta_J(\varepsilon))} \Gamma(2A(s+a)-j)\cos (\pi A(s+a)+k\pi)  \widehat{\varphi}(1-s)(hy)^{-s}\d s.
$$
Note that $\re (2A(s+a)-j)>0$ on $\re s= \theta_J(\varepsilon)$ for $1\le j\le J-1$. 
After a change of variable $z=2A(s+a)-j$ and moving the line of integration to $\re z= \frac14$, we derive that 
$$
I_j = \frac{1}{2A} \cdot \frac1{2\pi \text{i}} 
\int_{(\frac14)} \Gamma(z)\cos\big(\tfrac{1}{2}\pi z + (k+\tfrac{1}{2}j)\pi\big)  
\widehat{\varphi}\Big(1+a -\frac{z+j}{2A}\Big) (hy)^{a-(z+j)/(2A)} \d z.
$$
In virtue of the formula: for $0<c<\frac12$ and $\alpha\in\C$, 
\begin{eqnarray*}
\frac1{2\pi \text{i}} \int_{(c)} \Gamma(s) \cos\big(\tfrac{1}{2}\pi s+\alpha\big) y^{-s} \d s = \cos (y+\alpha),
\end{eqnarray*}
we replace $\widehat{\varphi}$ with its inverse Mellin transform and interchange the integrals, and consequently obtain
$$
I_j =\frac1{2A} \int_0^\infty (hyu)^{a-j/(2A)} \varphi(u)\cos\big((hyu)^{1/(2A)} + (k+\tfrac{1}{2}j)\pi\big) \d u.
$$
This completes the proof.
\end{proof}

\begin{lemma}\label{lem3}
Let $\tau, \rho,\theta \in\R$, and 
$$
K_{\tau,\rho} (t) = (1-|t|)(1+\tau \cos (2\rho t+\theta)),\quad \mbox{ $\forall$ $t\in [-1,1]$.}
$$
Then for any real $\upsilon$,
\begin{eqnarray*}
\int_{-1}^1K_{\tau,\rho}(t) {\rm e}^{{\rm i}2\upsilon t} \d t 
= \bigg(\frac{\sin\upsilon}{\upsilon}\bigg)^2 
+\frac{\tau {\rm e}^{{\rm i}\theta}}2 \bigg(\frac{\sin(\upsilon+ \rho)}{\upsilon+\rho}\bigg)^2 
+ \frac{\tau{\rm e}^{-{\rm i}\theta}}2\left(\frac{\sin (\upsilon -\rho)}{\upsilon-\rho}\right)^2.
\end{eqnarray*}
\end{lemma}

This follows from 
$$
\int_{-1}^1 (1-|t|) \text{e}^{\text{i}2\upsilon t} \d t = \bigg(\frac{\sin \upsilon}{\upsilon}\bigg)^2.
$$

\vskip 8mm

\section{Proof of Theorem~\ref{thm1}}\label{proof_thm1}

Here and in the sequel, for all the implied constants in the $O$- or $\ll$-symbols, 
we shall not indicate their dependence on the parameters in (A1) and (A2)  for simplicity. 
For $x\in\R$, we denote by the symbol $\lceil{x}\rceil_+$ the {\it smallest positive} integer greater than $x$. 

Let $J'$ be defined as in Lemma~\ref{lem2} and set
$$
J_0=J' + \lceil{2A(\sigma^*-\vartheta)+\mbox{$\frac12$}}\rceil_+\ge 2.
$$
Write
$$
\vartheta(j) := \vartheta+ \frac{j}{2A} \quad \mbox{ (thus $\displaystyle \frac{j}{2A}-a = \vartheta(j) -\text{i}\xi$)}
$$
and 
$$
S_{\varphi, j}(x)
:= \sum_{n\ge 1} \frac{b_n}{\mu_n^{\vartheta(j)-\text{i}\xi}}
\int_0^\infty \frac{\varphi(u)}{u^{\vartheta(j)}} u^{\text{i}\xi}\cos\big( (h\mu_n x u)^{1/(2A)} + (k+\tfrac{1}{2}j)\pi\big) \d u.
$$
We apply Lemma~\ref{lem2} to (\ref{Sphi}) with $J_0$ in place of $J$, so together with  (A1),
\begin{equation}\label{s41}
S_\varphi(x) 
= \frac{\omega}{2Ah} \sum_{0\le j\le J_0} e_j (hx)^{1-\vartheta(j)+\text{i}\xi } S_{\varphi, j}(x)
+ O_{\varepsilon}\big(L^{-1} x^{1-\vartheta-1/(2A)}\big)
\end{equation}
for any $x\in [X/2, X]$. 

It remains to estimate  $S_{\varphi, j}(x)$ in (\ref{s41}) for $1\le j\le J_0$ .
Replacing $u^{1/(2A)}$ by $w$,  the integral in $S_{\varphi, j}(x)$ equals the real part of a scalar multiple of 
$$
\int_0^\infty \frac{\varphi(w^{2A})}{w^{2A\vartheta(j)}} w^{2A-1+\text{i}2A\xi} \text{e}^{\text{i}(h\mu_n x)^{1/(2A)} w} \d w
=\int_0^\infty F(w) \text{e}^{\text{i}Yw} \d w
\quad
\text{(say)}.
$$
Clearly, for all $r\ge 0$ the support of $F^{(r)}(w)$ is contained in the interval $|w-1|\ll L^{-1}$ due to  the support of $\varphi$,
and $F^{(r)}(w)\ll_r \sum_{j=0}^r\|\varphi^{(j)}\|_\infty\ll_r L^r$, thanks to \eqref{derivation_varphi}. 
Thus a successive integration by parts shows that for any $r\ge 0$,  
the right-side is $\ll Y^{-r}\int_0^\infty |F^{(r)}(w)| \d w\ll_r Y^{-r}L^{r-1}\ll L^{-1} (L/Y)^r$, that is,
\begin{equation}\label{ibound}
\begin{aligned}
 \int_0^\infty \frac{\varphi(u)}{u^{\vartheta(j)}} u^{\text{i}\xi}\cos\big((h\mu_n x u)^{1/(2A)} + (k+\tfrac{1}{2}j)\pi\big) \d u
& \ll L^{-1} \bigg(\frac{L}{(\mu_n x)^{1/(2A)}}\bigg)^r 
\\
& \ll_{r} L^{-1}\delta^{-r}\mu_n^{-r/(2A)}.
\end{aligned}
\end{equation}
Choose $r(j)=\lceil{2A(\sigma^*-\vartheta(j))}\rceil_+$. It follows immediately that for $j\le J_0$,
$$
S_{\varphi, j}(x) 
\ll L^{-1}\delta^{-r(j)} \sum_{n\ge 1} \frac{|b_n|}{\mu_n^{\vartheta(j)+r(j)/(2A)}} \ll_{\delta} L^{-1}.
$$
Inserting into (\ref{s41}), we get immediately the desired results.

\vskip 8mm

\section{Proof of Theorem~\ref{thm2}}\label{sec_proof}

Let $\tau =\pm 1$ be selected up to our disposal. In view of (\ref{main}), we are led to consider 
\begin{eqnarray}\label{s5.0}
\tau \frac{|b_1|}{b_1\mu_1^{\text{i}\xi}} S_{\varphi, 0}(x) + O\big(L^{-1} X^{-1/(2A)}\big) 
\end{eqnarray}
and find $x$ such that its real part is bounded below by $cL^{-1}$ for some constant $c>0$. 
We divide $S_{\varphi,0}$ into two subsums over $1\le n\le N$ and $ n>N$ respectively, 
\begin{eqnarray}\label{s5.1}
S_{\varphi,0} = S_{\varphi,0}^{\le N}+ S_{\varphi,0}^{> N}.
\end{eqnarray}
For $S_{\varphi,0}^{> N}$, we repeat the above argument with (\ref{ibound}) and 
a choice of $r=\lceil{2A(\sigma^*-\vartheta)}\rceil_+$ in order that
\begin{eqnarray}\label{s5.2}
S_{\varphi, 0}^{>N} (x) \ll   L^{-1}\delta^{-r} \sum_{n>N} \frac{|b_n|}{\mu_n^{\vartheta+r/(2A)}} \le 
\delta L^{-1}
\end{eqnarray}
for some suitably large $N=N(\delta)$.

Now we apply the method in \cite{HBT94}. Recall $k= \kappa+\text{i}\eta$  and put
\begin{eqnarray*}
K_{\tau,\rho}(t) = (1-|t|) \big\{1+\tau \cos\big(2\rho t+\kappa\pi\big)\big\},
\qquad \mbox{ $\forall$ $t\in [-1,1]$,}
\end{eqnarray*}
where $\rho := (h\mu_1)^{1/(2A)}\alpha$ and the parameter $\alpha$ will be specified later.  Then $K_{\tau,\rho}(t)\ge 0$ for all $t\in [-1,1]$ and 
\begin{eqnarray*}
\int_{-1}^1 K_{\tau,\rho} (t) \d t \le 2.
\end{eqnarray*}
Consider any $T\in [(2X)^{1/(2A)},  (3X)^{1/(2A)}]$. We have 
\begin{equation}\label{s5.3}
\int_{-1}^1 S_{\varphi,0}\big((T+2\alpha t)^{2A}\big) K_{\tau,\rho}(t) \d t 
= \int_{-1}^1 S_{\varphi,0}^{\le N}\big((T+2\alpha t)^{2A}\big) K_{\tau,\rho} (t) \d t + \mbox{RT},
\end{equation}
where the remainder term $\mbox{RT}$ satisfies $|\mbox{RT}|\le 2\delta L^{-1}$. Besides,
\begin{equation}\label{s5.3b}
\begin{aligned}
 \int_{-1}^1 S_{\varphi,0}^{\le N}\big((T+2\alpha t)^{2A}\big) K_{\tau,\rho} (t) \d t
& = 
\sum_{1\le n\le N} 
\frac{b_n}{\mu_n^{\vartheta-{\rm i}\xi}} I_n
\end{aligned}
\end{equation}
where
\begin{eqnarray*}
I_n := \int_0^\infty \frac{\varphi(u)}{u^{\vartheta}} u^{\text{i}\xi}
\int_{-1}^1 K_{\tau,\rho}(t)\cos\big((h\mu_n u)^{1/(2A)}(T+2\alpha t) + k\pi\big) \d t\d u.
\end{eqnarray*}

The inner integral inside $I_n$ equals 
\begin{equation*}
\frac12\left(\text{e}^{\text{i}((h\mu_n u)^{1/(2A)}T + k\pi)} \mathscr{I}_n^+(u)
+\text{e}^{-\text{i}((h\mu_n u)^{1/(2A)}T + k\pi)} \mathscr{I}_n^-(u)\right)
\end{equation*}
with
\begin{align*}
\mathscr{I}_n^\pm(u)
& := \int_{-1}^1 K_{\tau, \rho}(t)  \text{e}^{\pm \text{i}2\upsilon t} \d t 
\end{align*}
where $\upsilon = (h\mu_n u)^{1/(2A)}\alpha$. Let  $\varpi_n^{\pm}(u) := \alpha\big((h\mu_nu)^{1/(2A)}\pm (h\mu_1)^{1/(2A)}\big)=\upsilon\pm \rho$ and recall $\delta_*=1$ if $*$ holds and 0 otherwise. By Lemma~\ref{lem3}, we have
\begin{align*}
\mathscr{I}_n^\pm(u)
& \;= 
 \bigg(\frac{\sin \upsilon}\upsilon\bigg)^2 
+\frac{\tau {\rm e}^{{\rm i}\kappa\pi}}2 \bigg(\frac{\sin \varpi_n^{\pm}(u) }{\varpi_n^{\pm}(u) }\bigg)^2 
+ \frac{\tau{\rm e}^{-{\rm i}\kappa\pi}}2\left(\frac{\sin \varpi_n^{\mp}(u) }{\varpi_n^{\mp}(u) }\right)^2
\\
& \;= \delta_{n=1} \frac{\tau }{2} \bigg(\frac{\sin\varpi_1^{-}(u)}{\varpi_1^{-}(u)}\bigg)^2 
\text{e}^{\mp\text{i} \kappa\pi}
+ O\bigg(\frac{1}{\alpha^2\mu_n^{1/A}}\bigg)
\end{align*}
where we have tacitly used $\big|(\mu_1u)^{1/(2A)} -\mu_n^{1/(2A)}\big|\gg \mu_n^{1/(2A)}$ for $n\ge 2$ and ${|u-1|\ll L^{-1}}$
whenever $X$ is sufficiently large,  leading to $|\varpi_n^{\pm}(u)|\gg |\alpha|\mu_n^{1/(2A)}$ for all $n\ge 2$ and  $|\varpi_1^{+}(u)|\gg |\alpha|\mu_1^{1/(2A)}$. Thus the inner integral inside $I_n$ equals 
$$
\delta_{n=1} \frac{\tau }{2} \bigg(\frac{\sin\varpi_1^{-}(u)}{\varpi_1^{-}(u)}\bigg)^2 
\cos \big((h\mu_n u)^{1/(2A)}T + \mathrm{i}\eta \pi\big)
+ O\bigg(\frac{1}{\alpha^2\mu_n^{1/A}}\bigg).
$$
With \eqref{integral_varphi}, it follows that 
\begin{equation*}
I_n = \delta_{n=1} \frac{\tau }{2} I_1^* + O\bigg(\frac{1}{\alpha^2L\mu_n^{1/A}}\bigg)
\qquad
(n\ge 1),
\end{equation*}
where
$$
I_1^*
:= \int_0^\infty \frac{\varphi(u)}{u^{\vartheta}} 
\left(\frac{\sin\varpi_1^{-}(u)}{\varpi_1^{-}(u)}\right)^2 u^{\text{i}\xi}\cos\big((h\mu_1 u)^{1/(2A)}T+\text{i}\eta \pi\big) \d u.
$$
Thus for \eqref{s5.3b}, we obtain with $\tau^2=1$ that
\begin{equation}\label{s5.4}
\begin{aligned}
\tau \frac{|b_1|}{b_1\mu_1^{\text{i}\xi}} \int_{-1}^1 S_{\varphi,0}^{\le N}\big((T+2\alpha t)^{2A}\big) K_{\tau, \rho} (t) \d t
& =  \frac{|b_1|}{2 \mu_1^{\vartheta}} I_1^*
+ O\bigg(\frac{1}{\alpha^{2}L} \sum_{n=1}^N \frac{|b_n|}{\mu_n^{\vartheta +1/A}}\bigg).
\end{aligned}
\end{equation}
Again using the fact that $\varphi$ is supported on $[1-1/L, 1+1/L]$ and $L=\delta^{-1} X^{1/(2A)}$, 
we easily infer $\varpi_1^{-}(u)\ll \alpha L^{-1}$ and $u^{1/(2A)}T= T+O(\delta)$, and hence
\begin{equation}\label{s5.5} 
I_1^*
= \left\{\cos\big((h\mu_1 )^{1/(2A)}T+\text{i}\eta \pi\big)+O\big(\delta+ \alpha L^{-1}\big)\right\}  \int_0^\infty \varphi(u) \d u. 
\end{equation}

Recall $N=N(\delta)$ is determined in (\ref{s5.2}). Next  we  choose  a sufficiently  large $\alpha= \alpha(\delta)$ so that the $O$-term in (\ref{s5.4}) is $\ll \delta L^{-1}$, and then $\alpha L^{-1}\le \delta$ for $X\ge \alpha(\delta)^{2A}$. Let $X_0(\delta)\ge \alpha(\delta)^{2A}$ be specified later. Following from (\ref{s5.1})--(\ref{s5.5})  (and \eqref{integral_varphi}), we obtain
\begin{equation}\label{s5.6}
\begin{aligned}
& \tau \frac{ |b_1|}{b_1\mu_1^{\text{i}\xi}} \int_{-1}^1  S_{\varphi,0}\big((T+2\alpha t)^{2A}\big) K_{\tau, \rho} (t) \d t
\\
& = \frac{|b_1|}{2\mu_1^{\vartheta}}\cos \left((h\mu_1 )^{1/(2A)}T+\text{i}\eta \pi\right)  \int_0^\infty \varphi(u) \d u  
+ O(\delta L^{-1})
\end{aligned}
\end{equation}
provided $X\ge X_0(\delta)$. 

We choose $\delta$ small enough so that the $O$-term in (\ref{s5.6}) is less than 
$\frac14|b_1| \mu_1^{-\vartheta}L^{-1}$ in size. 
When $T= 2n\pi/(h\mu_1)^{1/(2A)}$ with $n\in\N$, the main term in (\ref{s5.6}) is equal to
$$
\frac{|b_1|}{2\mu_1^{\vartheta}} \cosh(\eta\pi) \int_0^\infty \varphi(u) \d u 
\ge \frac{|b_1|}{2\mu_1^{\vartheta}} L^{-1},
$$
by the first inequality of \eqref{integral_varphi}. 
Thus for all $X\ge X_0(\delta)$ and 
any $T= 2n\pi/(h\mu_1)^{1/(2A)}\in [(2X)^{1/(2A)}, (3X)^{1/(2A)}]$ ($n\in \N$),  there exists $T_\tau\in [T-2\alpha, T+2\alpha]$ such that 
\begin{align*}
\re \bigg(\tau \frac{2|b_1|}{b_1\mu_1^{\text{i}\xi}} S_{\varphi,0} (T_\tau^{2A})\bigg)
& \ge \int_{-1}^1 \re \left(\tau \frac{|b_1|}{b_1\mu_1^{\text{i}\xi}} S_{\varphi,0}\big((T+2\alpha t)^{2A}\big)\right) K_{\tau, \rho} (t) \d t
\\
& \ge \frac14 \frac{|b_1|}{\mu_1^{\vartheta}} L^{-1}
\end{align*}
To its end, we further enlarge $X_0(\delta)$ such that the $O$-term in (\ref{s5.0}) does not exceed $\frac16|b_1| \mu_1^{-\vartheta}L^{-1}$ for $X\ge X_0(\delta)$, and $X_0(\delta)\ge (2\alpha(\delta))^{2A}$ (so that $T>2\alpha$). A little manipulation gives the desired inequalities, completing the proof. 

\vskip 8mm

\section{Proof of Theorem~\ref{thm2.5}}\label{proof_thm2.5}

Let us begin with an evaluation for $M_\varphi(x)$ defined in \eqref{M}.

\begin{lemma}\label{lem4} 
Let $\Upsilon$ be the set of all poles of $\phi(s)$, so $\Upsilon\subset \mathcal{D}$ is finite. 
With the notation in Theorem~\ref{thm2}, 
there exists a positive integer $M$ such that
for any large $x\ge X_0(\delta)$ and any $y$ with $|y-x|\le c_0 x^{1-1/(2A)}$, we have
$$
y^{-{\rm i} \xi} M_\varphi(y) 
= \sum_{\upsilon\in \Upsilon} \sum_{0\le j<{\rm ord}_{\upsilon}} \sum_{1\le m\le M}   
c_{j, m}(\upsilon) y^{\upsilon-\mathrm{i}\xi} (\log y)^j L^{-m} 
+ O\big(y^{1-\vartheta} L^{-2}\big),
$$
where ${\rm ord}_{\upsilon}$ is the order of $\phi$ at the pole $\upsilon$, 
the coefficients $c_{j, m}(\upsilon)\in \C$ are independent of $x$ and $y$, 
and $L:=\delta^{-1} x^{1/(2A)}$.
\end{lemma}

\begin{proof}
Applying the residue theorem to the integral in \eqref{M}, we obtain
$$
M_\varphi(y) 
= \sum_{\upsilon\in \Upsilon} \sum_{\substack{i\ge 0, \, j\ge 0\\ i+j<{\rm ord}_{\upsilon}}}
d_{i, j}(\upsilon) y^{\upsilon} \widehat{\varphi}^{(i)}(\upsilon) (\log y)^j
$$
for some coefficients $d_{i, j}(\upsilon)\in \C$.
By Taylor's theorem, we have the expansion
\begin{equation}\label{s6.1}
\begin{aligned}
\widehat{\varphi}^{(i)}(\upsilon) 
& = \frac{1}{L} \int_{-1}^1 \varphi_0(u) (1+uL^{-1})^{\upsilon-1} \log^i(1+uL^{-1}) \d u
\\
& = \sum_{1\le m\le M} d_{i, m}'(\upsilon) L^{-m} + O_{\upsilon, i, M}(L^{-M-1}).
\end{aligned}
\end{equation}
Since $L\asymp x^{1/(2A)}\asymp y^{1/(2A)}$, 
we take $M$ sufficiently large in \eqref{s6.1} and insert into the preceding formula. 
The lemma follows. 
\end{proof}

Next we prove the claim: Either condition of (a)-(c) in Theorem~\ref{thm2.5} implies that
\begin{eqnarray}\label{s6.2}
\lambda_{n+1}-\lambda_n\ge (4\delta +2c_0)\lambda_n^{1-1/(2A)}
\end{eqnarray}
holds only for finitely many $n$'s, where  $\delta$ and $c_0$ are the constants in Theorem~\ref{thm2}.

Consider any sufficiently large $\lambda_n\ge X_0(\delta)$ and take $x= (\lambda_n+\lambda_{n+1})/2$, the mid-point of $\lambda_n$ and $\lambda_{n+1}$. Since $\lambda_n< x$ and $0<2A<1$,  $\lambda_n^{1-1/(2A)}> x^{1-1/(2A)}$. If \eqref{s6.2} holds, then 
$$
\lambda_n< x-  (2\delta+c_0)x^{1-1/(2A)}< x+  (2\delta+c_0)x^{1-1/(2A)}<  \lambda_{n+1}.
$$
By Theorem~\ref{thm2} with this choice  of $x$ and observing 
$$
\lambda_n< x_\pm - 2\delta x^{1-1/(2A)} < x_\pm + 2\delta x^{1-1/(2A)} < \lambda_{n+1},
$$
we infer that $S_\varphi(x_\pm)=-M_\varphi(x_\pm)$ because $\varphi(\lambda_m/x_\pm)=0$ for all $m$, and  that
\begin{eqnarray}\label{s6.3}
\mp \re\bigg(\varsigma^{-1} \frac{M_\varphi(x_\pm)}{ ( \mu_1hx_\pm)^{{\rm i}\xi}}\bigg)
> c_1 x^{1-\vartheta}L^{-1} 
\end{eqnarray}
for some  constant $c_1>0$.

Now let us invoke  individually the conditions in Theorem~\ref{thm2.5}.

(a) 
Assume $\phi(s)$ has no pole. Then $M_\varphi(x)\equiv 0$ and thus \eqref{s6.3} cannot happen. 

(b) 
Suppose $\xi=0$ and $\Upsilon\subset \R$. Lemma~\ref{lem4} gives
$$
\re\bigg( \frac{\varsigma^{-1} M_\varphi(y)}{(\mu_1hy)^{{\rm i}\xi}}\bigg)
= \sum_{\upsilon\in \Upsilon} \sum_{0\le j<{\rm ord}_{\upsilon}} \sum_{1\le m\le M}   
c_{j, m}'(\upsilon)\, y^{\upsilon} (\log y)^j L^{-m} 
+ O\big(y^{1-\vartheta} L^{-2}\big),
$$
where $c_{j, m}'(\upsilon)\in\R$, $L=\delta^{-1} x^{1/(2A)}$ and $|y-x|\ll x^{1-1/(2A)}$. 
For such $y$ and  large $x$, the multiple sum on the right can be expressed as
$$ 
Q(x)
:= \big\{1+O\big(x^{-1/(2A)}\big)\big\} 
\sum_{\upsilon\in \Upsilon} \sum_{0\le j<{\rm ord}_{\upsilon}} \sum_{1\le m\le M} 
c_{j, m}'(\upsilon)\, x^{\upsilon} (\log x)^j L^{-m}.
$$
As $y^{1-\vartheta} L^{-2}=o(x^{1-\vartheta} L^{-1})$,  we infer by \eqref{s6.3} that $\mp Q(x) > \frac12 c_1x^{1-\vartheta}L^{-1}$ occurs simultaneously for the same $x$, which is impossible. 

(c)
Suppose $\re \upsilon< \vartheta_0=1-\vartheta+1/(2A)$ for all $\upsilon\in \Upsilon$. 
In this case, Lemma~\ref{lem4} implies
$$
\re\bigg( \frac{\varsigma^{-1} M_\varphi(y)}{ ( \mu_1hy)^{{\rm i}\xi}}\bigg)
=\widetilde{Q}(y)+ O\big(y^{1-\vartheta} L^{-2}\big),
$$
where
$$
\widetilde{Q}(y)
= \sum_{\upsilon\in \Upsilon} \sum_{0\le j<{\rm ord}_{\upsilon}} \sum_{1\le m\le M} 
c_{j, m}''(\upsilon)\, y^{\re\upsilon} (\log y)^j L^{-m} 
\cos\big((\im\upsilon-\xi)\log y + \theta_{j, m}(\upsilon)\big)
$$
for some real constants $c_{j, m}''(\upsilon)$ and $\theta_{j, m}(\upsilon)$. 
As $|y-x|\ll x^{1-1/(2A)}$, it follows that
$$
\widetilde{Q}(y)= \widetilde{Q}(x) + O\big(x^{\upsilon_{*}-1/(2A)}(\log x)^M  L^{-1}\big)
$$
where $\upsilon_{*} := \max\{\re\upsilon \, : \, \upsilon\in \Upsilon\}<\vartheta_0$. 
As $\upsilon_{*}-1/(2A)< 1-\vartheta$, from \eqref{s6.3} we deduce  concurrently the two inequalities $\mp \widetilde{Q}(x) > \frac12 c_1x^{1-\vartheta}L^{-1}$ once $x$ is large enough, a contradiction.

In summary, we have shown that \eqref{s6.2} cannot hold for any sufficiently large $\lambda_n$; in other words, for some $n_0$ and for  all  $n\ge n_0$,
$$
 \lambda_{n+1}-\lambda_n \ll  \lambda_n^{1-1/(2A)}
$$
with the implied constant independent of $n$. 

\vskip 3mm

\noindent 
{\bf Acknowledgements.} 
Lau is supported by GRF 17302514 of  the Research Grants Council of Hong Kong.
Liu is supported in part by NSFC grant 11531008, 
and Liu and Wu are supported in part 
by IRT1264 from the Ministry of Education.


\vskip 8mm

\begin{thebibliography}{C}

\bibitem{AS}
M. Asgari and R. Schmidt,
\textit{Siegel modular forms and representations},
Manuscripta Math. {\bf 104} (2001), 173--200. 

\bibitem{Berndt1971}
B. C. Berndt,
\textit{Identities involving the coefficients of a class of Dirichlet series. V},
Trans. Amer. Math. Soc. 160 (1971), 139--156

\bibitem{ChandrasekharanNarasimhan1962} 
K. Chandrasekharan \& R. Narasimhan, 
\textit{Functional equation with multiple gamma factors and the average order of arithmetical functions},
Ann. Math. {\bf 76} (1962), 93--136.

\bibitem{CN1963} 
K. Chandrasekharan \& R. Narasimhan, 
\textit{The approximate functional equation for a class of zeta-functions},
Math. Ann. {\bf 152} (1963), 30--64.

\bibitem{CG} 
J. B. Conrey \& A. Ghosh,
\textit{On the Selberg class of Dirichlet series: small degrees},
 Duke Math. J.  {\bf 72}  (1993),  673--693.

\bibitem{Hafner1981} 
J. L. Hafner, 
\textit{On the representation of the summatory functions of a class of arithmetical functions},
Analytic number theory (Philadelphia, Pa., 1980),  Lecture Notes in Math., 899, Springer, Berlin-New York, 1981, pp. 148--165. 

\bibitem{HBT94} 
D. R. Heath-Brown \& K.-M. Tsang, 
\textit{Sign changes of $E(T)$, $\Delta(x)$, and $P(x)$},
J. Number Theory {\bf 49} (1994), 73--83.

\bibitem{Kim2003}
H. H. Kim,
\textit{Functoriality for the exterior square of $GL_4$ 
and symmetric fourth of $GL_2$},
With appendix 1 by D. Ramakrishnan and 
appendix 2 by H. Kim \& P. Sarnak,
J. Amer. Math. Soc. {\bf 16} (2003), 139--183.

\bibitem{KimSarnak2003}
H. H. Kim \& P. Sarnak, 
\textit{Refined estimates towards the Ramanujan and Selberg conjectures},
J. Amer. Math. Soc. {\bf 16} (2003), no. 1, 175--183.
Appendix 2 to \cite{Kim2003}.

\bibitem{Lau1999}
Y.-K. Lau, \textit{On the mean square formula of the error term for a class of arithmetical functions}, 
Monatsh. Math.  {\bf 128}  (1999),  111--129. 

\bibitem{LauLiuWu2012}
Y.-K. Lau, J.-Y. Liu \& J. Wu,
\textit{Sign changes of the coefficients of automorphic $L$-functions},
Number Theory: Arithmetic In Shangri-La (Hackensack, NJ) 
(Shigeru Kanemitsu, Hongze Li, and Jianya Liu, eds.),
World Scientific Publishing Co. Pte. Ltd., 2013, pp. 141--181.

\bibitem{LT2002}
Y.-K. Lau \& K.-M. Tsang,
\textit{Large values of error terms of a class of arithmetical functions}, 
J. Reine Angew. Math.  {\bf 544}  (2002), 25--38. 

\bibitem{LiuQuWu2011} 
J.-Y. Liu, Y. Qu \& J. Wu,
\textit{Two Linnik-type problems for automorphic $L$-functions}, 
Math. Proc. Camb. Phil. Soc. {\bf 151} (2011), 219--227.

\bibitem{LiuWu2015}
 J.-Y. Liu \& J. Wu,
\textit{The number of coefficients of automorphic $L$-functions for $GL_m$ of the same signs},
J. Number Theory {\bf 148} (2015), 429--450.

\bibitem{LuoRudnickSarnak1999}
W.-Z. Luo, Z. Rudnick \& P. Sarnak,
\textit{On the generalized Ramanujan conjecture for $GL_m$},
Proceedings of Symposia In Pure Mathematics,
vol. {\bf 66}, part 2, 1999, 301--310.

\bibitem{MatomakiRadziwill2015}
K. Matom\"aki \& M. Radziwi\l\l,
\textit{Sign changes of Hecke eigenvalues},
GAFA, to appear.

\bibitem{MeherMurty2014}
J. Meher \& M. Ram Murty,
\textit{Sign changes of Fourier coefficients of half-integral weight cusp forms},
Inter. J. Number Theory {\bf 10} (2014), no. 4, 905--914.

\bibitem{Miz}
S. Mizumoto,
\textit{Poles and residues of standard L -functions attached to Siegel modular forms},
Math. Ann.  {\bf 289}  (1991),  589--612. 

\bibitem{Mol}
G. Molteni,
\textit{A note on a result of Bochner and Conrey-Ghosh about the Selberg class},
Arch. Math. (Basel)  {\bf 72}  (1999),  219--222. 

\bibitem{RSW}
E. Royer, J. Sengupta, J. Wu, 
\textit{Sign changes in short intervals of coefficients of spinor zeta function of a Siegel cusp form of genus 2},
Int. J. Number Theory  {\bf 10}  (2014),  327--339.

\bibitem{RudnickSarnak1996}
Z. Rudnick \& P. Sarnak,
\textit{Zeros of principal $L$-functions and random matrix theory},
Duke Math. J. {\bf 81} (1996), 269--322.

\bibitem{Sch}
R. Schmidt,
\textit{On the Archimedean Euler factors for spin $L$-functions},
Abh. Math. Sem. Univ. Hamburg  {\bf 72}  (2002), 119-43

\bibitem{Shi} 
G. Shimura,
\textit{On modular forms of half integral weight},
Ann. of Math.  {\bf 97}  (1973), 440--481. 
\end{thebibliography}
\end{document}